\documentclass[11pt]{amsart}

\usepackage[pdftex]{graphicx}
\usepackage{color}

\usepackage[centering,bottom=1.3in]{geometry}

\usepackage{amsfonts}
\usepackage{amsmath}
\usepackage{amssymb}
\usepackage{amsthm}

\usepackage{paralist}
\usepackage[colorlinks, citecolor=magenta,pagebackref=true,pdftex]{hyperref}

\usepackage[T1]{fontenc}

\newcommand\Defn[1]{\emph{\color{blue}#1}}

\renewcommand\emptyset{\varnothing}
\newcommand\Z{\mathbb{Z}}               
\newcommand\nZ{\Z_{\ge0}}               
\newcommand\R{\mathbb{R}}               
\newcommand\x{\mathbf{x}}
\newcommand\rEuler{\widetilde\chi}
\newcommand\rH{\widetilde{H}}

\newcommand\K{{\mathcal{K}}}
\newcommand\bK{{\mathcal{R}}}

\renewcommand\P{\mathcal{P}}

\newcommand\B{\mathcal{B}}

\renewcommand\phi{\varphi}

\newcommand\LatEnu{\mathrm{F}}

\DeclareMathOperator{\lk}{lk}
\DeclareMathOperator{\st}{st}
\DeclareMathOperator*{\relint}{relint}
\DeclareMathOperator*{\Int}{int}
\DeclareMathOperator*{\aff}{aff}

\DeclareMathOperator*{\cone}{cone}

\newtheorem{thm}{Theorem}[section]
\newtheorem{cor}[thm]{Corollary}
\newtheorem{lem}[thm]{Lemma}

\theoremstyle{definition}

\newtheorem{example}[thm]{Example}

\title{An Alexander-type duality for valuations}

\author{Karim Adiprasito}
\author{Raman Sanyal}
\address{Fachbereich Mathematik und Informatik, %
Freie Universit\"at Berlin, %
Germany}
\email{\{adiprasito,sanyal\}@math.fu-berlin.de}

\date{\today}
\thanks{K.~Adiprasito has been supported by DFG within the research training
group ``Methods for Discrete Structures'' (GRK1408) and by the Romanian NASR,
CNCS -- UEFISCDI, project PN-II-ID-PCE-2011-3-0533.}
\thanks{R.~Sanyal has been supported by European Research Council under the
European Union's Seventh Framework Programme (FP7/2007-2013) / ERC grant
agreement n$^\mathrm{o}$ 247029.}

\begin{document}

\begin{abstract}
 We prove an Alexander-type duality for valuations for certain
 subcomplexes in the boundary of polyhedra. These strengthen and simplify
 results of Stanley (1974) and Miller-Reiner (2005). We give a generalization
 of Brion's theorem for this relative situation and we discuss the topology of
 the possible subcomplexes for which the duality relation holds.
\end{abstract}

\maketitle

\section{Introduction}
\newcommand\F{\mathcal{F}}

Let $P \subset \R^d$ be a convex polytope with vertices in $\Z^d$ and let $q
\in \R^d$. Viewing $q$ as a light source, let $B \subseteq \partial P$ be 
the collection of points in the boundary of $P$
visible from $q$ --   the \emph{bright side} of $P$. 
That is, $B$ is the set of points $p \in \partial P$ for
which the open segment $(q,p)$ does not meet the relative interior of $P$.
Sticking to these figurative terms, let $D$ be the closure of the set of
\emph{dark points} $\partial P {\setminus}
B$. Stanley~\cite{stanley74} showed that for integral $n \ge 1$ the function
\[
    E_{P,B}(n) \ := \ | n \cdot (P {\setminus} B) \cap \Z^d |
\]
is the restriction of a univariate polynomial (and, by abuse of notation,
identified with that polynomial), and that
\begin{equation}\label{eqn:rd}
 (-1)^{\dim P} E_{P,B}(-n) \ = \ | n \cdot (P {\setminus} D) \cap \Z^d|
 \quad\text{for all } n \ge 1.
\end{equation}
By choosing $q \in \relint P$, we have that $(B,D) = (\emptyset,\partial P)$
and \eqref{eqn:rd} reduces to the well-known Ehrhart-Macdonald
reciprocity~\cite{macdonald}; see~\cite{BR07} for details. The set $B
\subseteq \partial P$ is a particular case of what Ehrhart~\cite{ehr1,ehr2}
calls a \Defn{reciprocal domain}, that is, a domain for
which~\eqref{eqn:rd} holds.

For a subset $S \subset \R^{d+1}$, the 
\Defn{lattice point enumerator} of $S$
is the multivariate Laurent series
\[
    \LatEnu_S(\x) \ := \  \sum_{a \in S \cap \Z^{d+1}} \x^a
\]
where $\x^a = x_1^{a_1} x_2^{a_2} \cdots x_{d+1}^{a_{d+1}}$. If we associate
to $P$ the pointed cone $C(P) := \cone( P \times \{1\})$ $\subset
\R^{d+1}$, then $\LatEnu_{C(P)}(\x)$ records the individual lattice points $(a,n) \in
\Z^{d+1}$ for
which $a \in n P$.  Stanley~\cite[Prop.\ 8.3]{stanley74} actually proved the
stronger result that
\begin{equation}\label{eqn:F_rd}
    (-1)^{\dim P} \LatEnu_{C(P{\setminus} B)}\left(\tfrac{1}{\x}\right) \ = \ \LatEnu_{C(P
    {\setminus} D)}(\x).
\end{equation}
where $\tfrac{1}{\x} = (\tfrac{1}{x_1}, \tfrac{1}{x_1}, \dots, \tfrac{1}{x_{d+1}})$.

The relation~\eqref{eqn:F_rd} holds for general rational pointed polyhedral
cones $C$ but not for arbitrary subsets in the boundary of $C$.  To see this,
we can choose $B$ as two non-adjacent triangles in the boundary of a
$3$-dimensional pyramid; one can check that $B$ is not a reciprocal domain.
The question which subsets in the boundary of $C$ are reciprocal domains was
investigated by Miller and Reiner~\cite{MR06}. They showed that the conditions
giving rise to reciprocal domains are topological rather than geometric in
nature. Let $C \subset \R^{d+1}$ be a rational, pointed polyhedral cone and
let $\Delta$ be a full-dimensional subcomplex of the boundary of $C$, i.e.\
$\Delta$ is a polyhedral complex induced by a collection of facets of $C$. Let
$\Delta^\prime$ be the subcomplex generated by the facets $F \not\in \Delta$.
Their result is
\begin{thm}[{\cite[Thm.~1]{MR06}}]\label{thm:MR1}
    If $\Delta$ is a Cohen-Macaulay complex, then
    \begin{equation}\label{eqn:MR1}
        (-1)^{d+1} \LatEnu_{C {\setminus} |\Delta|}\left(\tfrac{1}{\x}\right) \ = \ \LatEnu_{C
        {\setminus} |\Delta^\prime|}(\x).
    \end{equation}
\end{thm}
The proof of Theorem~\ref{thm:MR1} in~\cite{MR06} is given in terms of
combinatorial commutative algebra and relies on a connection between lattice
point enumerators and Hilbert series of $\Z^d$-graded modules. 

In this paper we give a simple proof of~\eqref{eqn:rd} and~\eqref{eqn:MR1}
that generalizes to a broader class of geometric objects and to valuations
other than counting lattice points (see Theorem~\ref{thm:atdvc}). Our proof
relies on basic facts from topological combinatorics and, as a byproduct,
gives a slightly more general class of complexes for which~\eqref{eqn:rd}
holds. Like Theorem~\ref{thm:MR1}, our results are reminiscent of Alexander
duality and we will emphasize this relation throughout.

The paper is organized as follows. In Section~\ref{sec:basics}, we recall
the notions of $\Lambda$-polytopes and valuations as well as (weakly)
Cohen-Macaulay complexes. In Section~\ref{sec:main} we state and prove an
Alexander-duality type relation which contains Thm.~\ref{thm:MR1} as a
special case. In Section~\ref{sec:rel_brion}, we give a relative version
of Brion's theorem which is interesting in its own right and highlights
the role played by weakly Cohen-Macaulay complexes. In
Section~\ref{sec:top} we focus on the topology of full-dimensional
(weakly) Cohen-Macaulay complexes in the boundary of spheres.  The bright
side $B$ of $P$ is homeomorphic to a ball of dimension $\dim P - 1$ and
thus Cohen-Macaulay. A natural question, which was answered affirmatively
in~\cite{MR06}, is if there exist full-dimensional Cohen-Macaulay
complexes in the boundary of polytopes that are not balls. We will extend
this result and we discuss possibly counterintuitive instances for
which~\eqref{eqn:rd} and~\eqref{eqn:MR1} apply.

\section{$\Lambda$-polytopes, valuations, and weakly Cohen-Macaulay complexes}
\label{sec:basics}

We start by setting the stage for the use of more general geometric objects
and valuations, following McMullen~\cite{mcmullen}.  Throughout, let $\Lambda
\subset \R^d$ be a fixed, full-dimensional discrete lattice or a vector space
over some subfield of $\R$.  We denote by $\P = \P(\Lambda)$ the collection of
polytopes in $\R^d$ with vertices in $\Lambda$.
A \Defn{$\Lambda$-valuation} is a map $\phi$ from $\P$ into
some abelian group such that
\[
    \phi(P \cup Q) \ = \ \phi(P) + \phi(Q) - \phi(P \cap Q)
\]
whenever $P \cup Q \in \P$ (and hence $P \cap Q \in \P$) and such that $\phi(t
+ P) = \phi(P)$ for all $t \in \Lambda$. We can extend $\phi$ to
\emph{half-open} polytopes as follows. If $B \subset \partial P$ is a the
union of facets $F_1,F_2,\dots,F_m$ of $P$, then
\[
\phi(P {\setminus} B) \ := \ \sum_{J \subseteq [k]} (-1)^{|J|}\,\phi(F_J)
\]
where $F_J := \bigcap\{ F_j : j \in J\}$.
In particular, if $B = \partial P$, we get
\begin{equation}\label{eqn:relint}
    \phi(\relint P) \ = \ \sum_{F \subseteq P} (-1)^{\dim P - \dim F}\phi(F)
\end{equation}
where the sum is over all non-empty faces $F$ of $P$. The following is the
basis for our considerations.
\begin{thm}[{\cite{mcmullen}}]\label{thm:rec}
    If $\phi$ is a $\Lambda$-valuation, then for all $n \in \nZ$
    \[
        \phi_P(n) \ := \ \phi(nP) 
    \]
    agrees with a univariate polynomial of degree $\le \dim P$ and
    \[
        (-1)^{\dim P} \phi_P(-1) \ = \ \phi( \relint(-P) ).
    \]
\end{thm}
A \Defn{$\Lambda$-complex} is a polyhedral complex $\K$ such that every face
is a $\Lambda$-polytope.  The complex is \Defn{pure} if all inclusion-maximal
faces have the same dimension.  For example, the collection of proper faces of
a $\Lambda$-polytope $P$ is a pure $\Lambda$-complex, called the
\Defn{boundary complex} $\B(P)$.  The underlying set of $\K$ is denoted by
$|\K|$ and, since this is the disjoint union of relatively open polytopes, we
can extend $\phi$ to $\Lambda$-complexes by setting
\[
    \phi(|\K|) \ := \  \sum_{F \in \K} \phi(\relint F)
\]

For a given face $F$ in a
polyhedral complex $K$, the \Defn{link} of $F$ in $K$ is the polyhedral
subcomplex
\[
    \lk_\K(F) \ = \ \{ G \in \K : G \cap F = \emptyset, G \cup F \subseteq H \in
    \K \}.
\]
For a subcomplex $\Delta \subset \K$, a face $F \in \Delta$ is an
\Defn{interior face} of $\Delta$ if $\lk_\K(F) \subset \Delta$ and a
\Defn{boundary face} otherwise. The boundary of $\Delta$ is the subcomplex
$\partial\Delta$ of all boundary faces. Note that for $F \not\in \Delta$, we
have $\lk_\Delta(F) = \emptyset \not= \{\emptyset\}$ with reduced Euler
characteristic $\rEuler(\emptyset) = 0$.

A pure complex $\K$ is \Defn{weakly Cohen-Macaulay} if 
\[
    \rH_i(\lk_\K(F)) \ = \ 0 \qquad \text{for all } 0 \le i < \dim
    \lk_\K(F).
\]
for all non-empty faces $F \in K$.  Thus $\K$ is Cohen-Macaulay if
additionally $\rH_i(\K) = 0$ for all $0 \le i < \dim \K$. This is a stronger 
condition as, for instance, weakly Cohen-Macaulay complexes
are not necessarily connected. Since $G \subseteq F$
implies $\lk_\K(F) \subseteq \lk_\K(G)$, we get that $K$ is weakly
Cohen-Macaulay if and only if every vertex link of $K$ is Cohen-Macaulay.
Munkres~\cite{munkres} proved that Cohen-Macaulayness of a complex $K$ is a
topological property of the underlying pointset $|K|$ and hence $K$ is weakly
Cohen-Macaualy if $\rH_i(|K|,|K| {\setminus} p)$ vanishes for $i < \dim K$.
Note that what we define is the notion of (weakly) $\Z$-CM complexes as our
ring of coefficients is $\Z$ throughout (cf.~\cite[Sect.~11]{bjorner}); 
however, most of our results hold for general rings of coefficients.

Finally, a pure $\Lambda$-complex $\K$ of dimension $d$ is a \Defn{homology
manifold}, if for every face $F$ of $\K$, the reduced homology of $\lk_\K(F)$
is identically zero or if
\[
    \rH_\ast(\lk_\K(F)) \ \cong \ \rH_\ast(S^{d-\dim F + 1}).
\]
In particular, if $|\K|$ is a manifold, then $\K$ is a homology manifold, and every homology manifold is weakly CM.

\section{An Alexander-type duality}
\label{sec:main}

In this section we prove Alexander-type duality relations for
$\Lambda$-valuations that relates complementary complexes $\Delta$ and
$\Delta^\prime$ inside $\Lambda$-complexes.

\begin{thm}[Alexander-type duality for valuations]\label{thm:atdvc}
    Let $\K$ be a $d$-dimensional $\Lambda$-complex such that $\K$ is a
    homology manifold with boundary and let $B \subset \partial \K$ be a
    full-dimensional, weakly Cohen-Macaualay subcomplex. Let $D$ be the
    closure of $\partial P {\setminus} B$. If $\phi$ is a $\Lambda$-valuation,
    then for all $n \ge 1$
    \[
    (-1)^d \phi_{|\K| {\setminus} |B|}(-n) \ = \ \phi_{-(|\K| {\setminus} |D|)}(n)
    \]
    and
    \[
    (-1)^d \phi_{|\K| {\setminus} |B|}(0) \ = \ 
    \phi(\{0\}) \bigl(\rEuler(\K) - \rEuler(B))\bigr) \ = \
    \phi_{-(|\K| {\setminus} |D|)}(0).
    \]
\end{thm}

For the proof of the theorem we need to relate the combinatorics of
inclusion-exclusion for the valuation $\phi$ to the topology of $\Delta$. The
main observation, captured in the following lemma, is that weakly
Cohen-Macaulay $(d-1)$-complexes which are embedded into the boundary of a
$d$-dimensional homology manifold are rather restricted. 
\begin{lem} \label{lem:wCM_Eulerian}
    Let $\bK$ be a $(d-1)$-dimensional homology manifold without boundary and
    let $\Delta \subseteq \bK$ be a pure, weakly Cohen-Macaulay subcomplex of
    full dimension $d-1$. Then for every $\emptyset \neq F \in \Delta$ 
    \[
        \rH_k(\lk_\Delta(F)) \ = \ 
        \begin{cases}
            \Z,& \text{ if $F$ is an interior face of dimension $d-k-2$, and}\\
            0,& \text{ otherwise.}
        \end{cases}
    \]
\end{lem}
In other words, a full-dimensional, weakly Cohen-Macaulay subcomplex of a
homology manifold is again a homology manifold. 
\begin{proof}
    The link $\lk_\Delta(F)$ is a subcomplex of $ L = \lk_\bK(F)$, which
    has the homology of a $k$-sphere. Thus, if $F$ is an interior face of
    $\Delta$, then $\lk_\Delta(F) = L$ and $\rH_*(\lk_\Delta(F)) =
    \rH_\ast(S^k)$.

    If $\lk_\Delta(F) \subsetneq  L$ is a proper subcomplex, it is
    sufficient to show that $\rH_k(\lk_\Delta(F)) = 0$ for $k = \dim
    \lk_\Delta(F)$, as $\Delta$ is weakly Cohen-Macaulay. For this 
    observe that $|L|{\setminus} |\lk_\Delta(F)|$ is non-empty.
    By Alexander duality for homology
    spheres~\cite[\S~72]{munkres_elements}, we get that 
    \[
        0 \  = \   \rH_{-1}(|L|{\setminus} |\lk_\Delta(F)|) \  = \  
        \rH_{k}(\lk_\Delta(F)).
    \]
    Alternatively, it is sufficient to show that $\lk_\Delta(F)$ is homotopic
    to a subcomplex of dimension $k-1$. To see this, note that $\lk_\Delta(F)$
    is a full-dimensional subcomplex of the $k$-dimensional homology manifold
    $L$. Thus, $\lk_\Delta(F)$ has a \emph{free face} and, using 
    Whitehead's language of cellular
    collapses~\cite{Whitehead}, $\lk_\Delta(F)$ collapses to a subcomplex of
    its $(k-1)$-skeleton. Since a collapse in particular provides a certificate for
    deformation retraction, this finishes the proof.
\end{proof}

\begin{proof}[Proof of Theorem~\ref{thm:atdvc}]
As a subset of $\R^d$, $|\K|$ is partitioned by the relative interiors of
faces $G \in \K$ and thus
\[
    \phi_{|\K|{\setminus} |B|}(n) \ = \ \sum_{G \in \K {\setminus} B}
     \phi_{\relint G}(n),
\]
For the case $n\neq 0$: as $\phi_{|\K|}(n) \ = \ \phi_{n|\K|}(1)$, is it is sufficient to prove the
claim for $n = -1$. From Theorem~\ref{thm:rec} and~\eqref{eqn:relint}, we get
\begin{align*}
    (-1)^d \phi_{|K|{\setminus}|B|}(-1)  
    &\  =  \  \sum_{G \in \K {\setminus} B}  (-1)^{d - \dim G}  \phi(-G)  \\
    &\  =  \  \sum_{G \in \K {\setminus} B} (-1)^{d - \dim G} \sum_{\sigma
    \subseteq G \text{ face}} \phi(\relint (-\sigma))  \\
    &\  =  \  \;\; \sum_{\sigma \in \K} W_\sigma\,\phi(\relint (-\sigma))
    \end{align*}
where for a face $\sigma \in \K$
\[
    W_\sigma \  := \  (-1)^d\sum_{\sigma \subseteq G \in \K {\setminus} B} (-1)^{\dim G}
    \  = \  (-1)^{d - \dim G}\bigl(\rEuler(\lk_\K(\sigma)) -
    \rEuler(\lk_B(\sigma)) \bigr)
\]
It follows from Lemma~\ref{lem:wCM_Eulerian}, that $W_\sigma = 1$ if $\sigma
\in \K {\setminus} D$ which proves the claim. The proof of the case $n=0$ is analogous.
\end{proof}

Since the boundary of every $\Lambda$-polytope is a sphere, we can extend the
validity of~\eqref{eqn:rd} to general $\Lambda$-valuations.
\begin{cor}
    Let $P \subset \R^d$ be a $\Lambda$-polytope. Let $B$ be the underlying space of
    a full-dimensional, weakly CM subcomplex and let
    $D$ be the closure of $\partial P {\setminus} B$.
    If $\phi$ is a $\Lambda$-valuation, then
    \[
    (-1)^{\dim P}\phi_{P {\setminus} B}(-n) \ = \ \phi_{-(P {\setminus} D)}(n)
    \ \ \text{for all}\ \ n\neq 0.
    \]
\end{cor}
This is indeed a generalization of~\eqref{eqn:rd}, as $\phi(S) = |S \cap
\Z^d|$ is invariant under automorphisms of the lattice $\Lambda = \Z^d$. We
give an example for a self-reciprocal domain, that is, $D = \mathrm{T}(B)
\subset \partial P$, where $\mathrm{T}$ is an automorphism of $\Lambda$ with
$T(P) = P$. 

\begin{example}
    Let $P = [0,1]^4 = P_1 \times P_2$ be the $4$-cube presented as the
    product of two squares $P_1 = P_2 = [0,1]^2$. The boundary of the $4$-cube
    contains a $2$-dimensional torus $T = \partial P_1 \times \partial P_2$,
    which decomposes $\partial P$ into two solid tori $S_1 = P_1 \times
    \partial P_2$ and $S_2 = \partial P_1 \times P_2$. As these are
    $3$-manifolds with boundary, both $S_1$ and $S_2$ are pure $3$-dimensional
    weakly Cohen-Macaulay subcomplexes. The Ehrhart function for a $k$-cube is
    $E_{[0,1]^k}(n) = (n+1)^k$. Thus the relative Ehrhart function is
    \[
        E_{P,S_1}(n) \ = \ (n+1)^4 - 4 (n+1)^3 + 4(n+1)^2  \ = \ n^4 - 2n^2 +
        1
    \]
    and $(-1)^4E_{P,S_1}(-n) = E_{P,S_2}(n) =
    E_{P,S_1}(n)$.\hfill$\diamond$
\end{example}

Towards a proof for Theorem~\ref{thm:MR1}, let us record the
following general lemma. For a polyhedral cone $C$ and a point $a \in C$, let
$\sigma_a \subset C$ be the unique face with $a \in \relint \sigma_a$.
\begin{lem}\label{lem:genF}
    Let $C \subset \R^{d+1}$ be a rational $(d+1)$-dimensional cone and $\Delta
    \subset \B(C)$ an arbitrary subcomplex. Then
    \[
    (-1)^{d+1} \LatEnu_{C {\setminus} |\Delta|}\left(\tfrac{1}{\x}\right)
    \ = \ \LatEnu_{\relint C}(\x) \ + \
    \sum_{a \in |\Delta| \cap \Z^{d+1}} (-1)^{d - \dim \sigma_a }\,
    \rEuler(\lk_\Delta(\sigma_a))\x^a.
    \]
\end{lem}
Notice that $\lk_\Delta(\sigma) \subseteq \lk_{\B(C)}(\sigma) \cong S^{d
 - \dim \sigma}$. Thus, the coefficient of $\x^a$ in the equation above is
the Euler characteristic of the \emph{Alexander dual} of
$|\lk_\Delta(\sigma_a)| \subset S^{d  - \dim \sigma_a}$.

\begin{proof}
    From Ehrhart theory (cf.~\cite[Prop.~7.1]{stanley74}), we have for a rational cone $G$
    \[
    (-1)^{\dim G} \LatEnu_G\left(\tfrac{1}{\x}\right) \ = \ \LatEnu_{\relint G}(\x).
    \]
    Thus, from
      \[
        \LatEnu_{C{\setminus}|\Delta|}(\x) \ = \ \LatEnu_C(\x) \ - \ 
        \sum_{G \in \Delta} \LatEnu_{\relint G}(\x)
    \]
    we obtain
    \[
    (-1)^{d+1} \LatEnu_{C {\setminus} |\Delta|}\left(\tfrac{1}{\x}\right)
    \ = \ \LatEnu_{\relint C}(\x) \ + \
    \sum_{G \in \Delta} (-1)^{d - \dim G} \LatEnu_G(\x)
    \]
    which shows that the right-hand side is supported on $\relint(C) \cup
    |\Delta|$.  Now for $a \in |\Delta| \cap \Z^{d+1}$, the coefficient of
    $\x^a$ on the right-hand side is 
    \[
    (-1)^{d+1}
    \sum_{\sigma_a \subseteq G \in \Delta} (-1)^{\dim G} \ = \
    (-1)^{d-\dim \sigma_a }
     \sum_{\overline{G} \in \lk_\Delta(\sigma_a)} (-1)^{\dim \overline{G}} 
     \ = \ (-1)^{d-\dim \sigma_a }\rEuler(\lk_\Delta(\sigma_a)),
    \]
    which proves the claim.
\end{proof}

\begin{proof}[Proof of Theorem~\ref{thm:MR1}]
    If $\Delta$ is Cohen-Macaulay, then for every face $F \in \Delta$, the
    link  $\lk_\Delta(F)$ has the reduced Euler characteristic of a $(d-1-\dim
    F)$-sphere if $F$ is interior and the reduced Euler characteristic of a
    point otherwise. Together with Lemma~\ref{lem:genF} this gives us
    \[
    (-1)^{d+1}\LatEnu_{C {\setminus} |\Delta|}\left(\tfrac{1}{\x}\right) \ = \ \LatEnu_{\relint
    C}(\x) \ + \ \sum_{a \in (|\Delta|{\setminus} |\Delta^\prime|)\cap \Z^{d+1}}
    \x^a \qedhere
    \]
\end{proof}

\section{A relative Brion theorem}
\label{sec:rel_brion}

In this section we give a version of Brion's theorem~\cite{brion} (see
also~\cite{bhs}) suitable in the presence of a \emph{forbidden} subcomplex. To
make our results more transparent, let us start with the classical Brion-Gram
relation for polytopes and an interesting complementary version. For
a subset $S \subseteq \R^d$, let us denote by $[S] : \R^d \rightarrow \{0,1\}$
the indicator function. Note, that $[S\cap T] = [S] \cdot [T]$.

Let $C = \{ x \in \R^{d+1} : \langle a_i, x\rangle \le 0 \text{ for } i =
1,2,\dots, m \}$ be a polyhedral cone. For a non-empty face $F \subseteq C$
let $I(F) = \{ i \in [m] : \langle a_i, x \rangle = 0 \text{ for all } x \in
F\}$ and
define the \Defn{tangent cone} of $C$ at $F$ as 
\[
T_C(F)  \ := \ \{ x \in \R^d : \langle a_i, x \rangle \le 0 \text{ for }
i \in I(F) \}
\]

\begin{lem}\label{lem:coneBG}
    Let $C \subset \R^{d+1}$ be a full-dimensional polyhedral cone.  Then
    \[
        \sum_{\emptyset \neq F \subseteq C} (-1)^{\dim F} [T_C(F)] \ = \
        (-1)^{d+1} [\Int(-C)]
    \]
\end{lem}
\begin{proof}
    If $p \in \Int(-C)$, then $p \in T_C(F)$ if and only if $F = C$.
    For $p \in \R^{d+1} {\setminus} \Int(-C)$, let $J = \{ i \in [m] :
    \langle a_i, p \rangle \le 0 \}$. Then
    \[
        C_J \ := \ \{ x \in \R^{d+1} : \langle a_i, x \rangle \le 0 \text{ for
        } i \in J \}
    \]
    is the product of a linear space and a pointed polyhedral cone and thus
    has Euler characteristic $\chi(C_J) = 0$.  Moreover, by sending the point
    $p \in C_J$ to infinity, the faces of $C_J$ are exactly those faces $F
    \subseteq C$ for which $p \in T_C(F)$ and the left-hand side of the stated
    equation computes the Euler characteristic of $C_J$.
\end{proof}

From that, we can deduce the usual Brianchon-Gram relation. If $P \subset
\R^d$ is polytope and $F$ a face, then the tangent cone of $P$ at $F$ is
defined analogously as above and, equivalently, $T_P(F) \ = \ q_F + \cone(P
- q_F) $, where $q_F \in \relint F$. In particular, we have $T_P(P) = \R^d$
and $T_P(\emptyset) = P$.

\begin{cor}\label{cor:BG}
    If $P \subset \R^d$ is a polytope, then
    \[
        [P] \ = \ \sum_{\emptyset \neq F \subseteq P} (-1)^{\dim F} [T_P(F)].
    \]
\end{cor}
\begin{proof}
    Let $C = C(P) \subset \R^{d+1}$ be cone associated to $P$. Let $H = \R^d
    \times \{1\}$. Then the $(k+1)$-faces $\hat F$ of $C$ bijectively
    correspond to $k$-faces under $F = \hat F \cap H$. With the
    appropriate identifications, $[P] = [C]\cdot [H]$ and, in particular, 
    $[T_P(F)] = [T_C(\hat F)]\cdot [H]$. Since $H \cap \Int(-C) = \emptyset$,
    we get from Lemma~\ref{lem:coneBG} 
    \[
        [P] \ = \ [H] \cdot [T_C(0)] \ = \ [H]
        \sum_{\{0\} \subsetneq \hat F \subseteq C} (-1)^{\dim \hat F-1}
        [T_C(\hat F)] 
        \ = \ \sum_{\emptyset \neq F \subseteq P} (-1)^{\dim F} [T_P(F)],
    \]
    which proves the claim.
\end{proof}

We also get an interesting  complementary version as follows. For every face
$F \subseteq P$, the tangent cone is of the form $T_P(F) = \aff(F) + C_P(F)$
where $C_P(F)$ is the unique cone contained in $\aff(F)^\perp$. Let us define
the \Defn{inverted tangent cone} as $T^{-1}_P(F) = \aff(F) - C_P(F)$.

\begin{cor}\label{cor:invBG}
    Let $P \subset \R^d$ be a full-dimensional polytope. Then
    \[
        (-1)^d [\relint(P)] \ = \ \
        \sum_{\emptyset \neq F \subseteq P} (-1)^{\dim F} [T^{-1}_P(F)].
    \]
\end{cor}
\begin{proof}
    Let $P = \{ x : \langle a_i, x\rangle \le b_i \text{ for } i \in [m]\}$.
    For a non-empty face $F \subseteq P$, the inverted tangent cone is
    given by
    \[
    T^{-1}_P(F) \ = \ \{ x \in \R^d : \langle a_i, x \rangle -b_i \ge 0 \text{
    for } i \in I(F)\}.
    \]
    Now consider $C = C(-P) = \{ (x,t) : t \ge 0, \langle -a_i,x\rangle - b_i t \le 0,
    i \in [m] \}$ and $H = \R^d \times \{-1\}$. Then, with appropriate
    identifications, $\relint(-C) \cap H = \relint(P)$ and $T_C(\hat F) \cap H
    = T^{-1}_P(F)$.  Lemma~\ref{lem:coneBG} now yields the result.
\end{proof}

For dealing with forbidden subcomplexes, we will also need the following
relative versions of the two Brianchon-Gram relations.
If $\Delta \subseteq \B(P)$ is a full-dimensional subcomplex of the boundary,
then this induces a subcomplex $\Delta_F \subseteq \B(T_P(F))$ in the tangent
cone of every face $F \subsetneq P$. This subcomplex is pure of dimension
$d-1$ or empty.  We write $T_{P,\Delta}(F) = T_P(F) {\setminus} |\Delta_F|$ for
the tangent cone minus the faces induced by $\Delta$, and
$T^{-1}_{P,\Delta}(F)$ for the analogously defined relative inverted tangent
cone.

\begin{lem}\label{lem:rel_BG}
    Let $P \subset \R^d$ be a $d$-polytope and $\Delta \subseteq \B(P)$ a 
    full-dimensional subcomplex. Let $\Delta^\prime \subseteq \B(P)$ be
    the subcomplex spanned by the facets not contained in $\Delta$. 
    Then
    \[
    [P {\setminus} |\Delta|] \  = \  \sum_{F} (-1)^{\dim F}
    [T_{P,\Delta}(F)]
    \]
    and
    \[
    (-1)^{\dim P} [P {\setminus} |\Delta^\prime|] \  = \  \sum_{F} (-1)^{\dim F}
    [T^{-1}_{P,\Delta}(F)]
    \]
    where the sums are over all non-empty faces $F \subseteq P$.
\end{lem}
\begin{proof}
    We prove only the first statement as the proof of the second relation is
    analogous.
    Let $p \in \R^d$ be an arbitrary point. If $p$ is not contained in the
    affine span of any face of $\Delta$, then $[{T_{P,\Delta}(F)}](p) =
    [T_P(F)](p)$ for all non-empty faces $F \subseteq P$ and the
    identity is Corollary~\ref{cor:BG}.  Thus, suppose that $p$ is
    contained in some hyperplane spanned by a facet in~$\Delta$.

    If $p \in P$, then the unique face $F \subseteq P$ containing $p$ in
    the relative interior is a face of $\Delta$. In this case
    $p \in T_{P,\Delta}(G)$ if and only if $G$ and $F$ are contained in
    a common face of $\Delta$. That is, if $D$ is contained in the
    \emph{closed star}
    $
        \st_\Delta(F) \ := \ \{ G \in \Delta: F \cup G \subseteq K \in \Delta
        \}
    $ of $F$ in $\Delta$.
    The right-hand side of the stated equation evaluated at $p$
    can be written as 
    \[
        \sum_{F \in \B(P){\setminus}\{\emptyset\}} (-1)^{\dim F} \ - \ \sum_{D
        \in \st_\Delta(F){\setminus} \{\emptyset\}} (-1)^{\dim D}.
    \]
    This is the difference of the unreduced Euler characteristics of two
    contractible complexes and therefore $0 = 1-1$.

    If $p \in \R^d {\setminus} P$, let $F_1,\dots, F_k \subseteq P$ be the
    $(d-1)$-dimensional faces of $\Delta$ for which $p$ is contained in the
    affine hyperplane $H_i := \aff(F_i)$  spanned by $F_i$. We have to show
    that 
    \begin{equation}\label{eqn:rel_BG}
        \sum\{(-1)^{\dim G} : p \in T_P(G) \text{ and } G \subseteq F_i \text{
        for some } i =1,\dots, k \}  \ = \ 0,
    \end{equation}
    as this is the collection of terms missing from the usual Brianchon-Gram.
    For $I \subseteq [k]$, let $F_I = \cap_{i \in I} F_i$ and define
    \[
        s_I \ := \ \sum\{(-1)^{\dim G} : p \in T_P(G) \text{ and } G \subseteq
        F_I \}.
    \]
    We can rewrite the left-hand side of~\eqref{eqn:rel_BG} as
    \[
        \sum_{\emptyset \neq I \subseteq [k]} (-1)^{|I|-1} s_I.
    \]
    But for a fixed $I$, we have that $s_I$ is equal to the left-hand side of
    the Brianchon-Gram relation applied to $F_I$ and a point $p \not\in F_I$
    inside $\aff(F_I)$. Thus $s_I = 0$.
\end{proof}

We can now state our generalization of Brion's theorem.

\begin{thm}[Relative Brion's theorem]
    Let $P \subset \R^d$ be a full-dimensional polytope with vertices $v_1,
    v_2, \dots, v_n \in \Z^d$.  Let $\Delta \subseteq \B(P)$ a pure and
    $d$-dimensional weakly
    Cohen-Macaulay subcomplex and let $\Delta^\prime \subseteq \B(P)$ be the
    subcomplex generated by the facets of $P$ not contained in $\Delta$.  Then
    \[
    \LatEnu_{P {\setminus} |\Delta|}(\x) \ = \ 
    \LatEnu_{T_{P,\Delta}(v_1)}(\x) \ + \ 
    \LatEnu_{T_{P,\Delta}(v_2)}(\x) \ + \ 
    \cdots \ + \ 
    \LatEnu_{T_{P,\Delta}(v_n)}(\x)
    \]
    and
    \[
    (-1)^d\LatEnu_{P {\setminus} |\Delta|}(\tfrac{1}{\x}) \ = \ \LatEnu_{-(P
    {\setminus} |\Delta^\prime|)}(\x).
    \]
\end{thm}
\begin{proof}
    The first statement follows from the same consideration as in~\cite{bhs}:
    Observe that for $S \subset \R^d$, we have
    $\LatEnu_S(x) = \sum_{a \in \Z^d} [S](a) \x^a$ and from Lemma~\ref{lem:rel_BG}
    we get
    \[
    \LatEnu_{P{\setminus}|\Delta|}(\x) \ = \ \sum_{F} 
    (-1)^{\dim F}\LatEnu_{T_{P,\Delta}(F)}(\x)
    \]
    where the sum is over all non-empty faces $F \subseteq P$. Now if $F$ is
    not a vertex, the relative tangent cone $T_{P,\Delta}(F)$ is not pointed,
    that is, $t + T_{P,\Delta}(F) = T_{P,\Delta}(F)$ for some $t \neq 0$. On
    the level of lattice point enumerators, this means
    $\x^t\LatEnu_{T_{P,\Delta}(F)}(\x) = \LatEnu_{T_{P,\Delta}(F)}(\x)$ and thus
    $\LatEnu_{T_{P,\Delta}(F)}(\x) = 0$. This proves the first statement.

    By the same token, we get from Lemma~\ref{lem:rel_BG}
    \[
    (-1)^d \LatEnu_{P{\setminus}|\Delta^\prime|}(\x) \ = \ \sum_{F} (-1)^{\dim
    F}\LatEnu_{T^{-1}_{P,\Delta}(F)}(\x)
    \]
    and thus
    \[
    (-1)^d \LatEnu_{P{\setminus}|\Delta^\prime|}(\x) \ = \ \sum_{i=1}^n 
    \LatEnu_{T^{-1}_{P,\Delta}(v_i)}(\x)
    \]
    Let us write $T_{P,\Delta}(v_i) = v_i + C_i {\setminus} |\Delta_{v_i}|$
    where $C_i$
    is a rational polyhedral cone and $\Delta_i = \Delta_{v_i}$ is the subcomplex
    induced by $\Delta$. In particular, $\LatEnu_{T_{P,\Delta}(v_i)}(\x) =
    \x^{v_i}\LatEnu_{C_i {\setminus} |\Delta_i|}(\x)$. Since $\Delta$ is weakly
    Cohen-Macaulay, we have that $\Delta_i$ is Cohen-Macaulay and  by
    Theorem~\ref{thm:MR1}
    \[
    (-1)^d \LatEnu_{T_{P,\Delta}(v_i)}(\tfrac{1}{\x}) \ = \ 
    \x^{-v_i}\LatEnu_{C_i {\setminus} |\Delta^\prime_i|}(\x) \ = \
    \LatEnu_{T^{-1}_{-P,-\Delta^\prime}(-v_i)}(\x)
    \]
    For the finishing touch, we calculate
    \[
    \LatEnu_{P {\setminus} |\Delta|}(\tfrac{1}{\x}) \ = \ 
    \sum_{i=1}^n \LatEnu_{T_{P,\Delta}(v_i)}(\tfrac{1}{\x}) 
    \ = \ 
    \sum_{i=1}^n (-1)^d
    \LatEnu_{T^{-1}_{-P,-\Delta^\prime}(-v_i)}(\x)
    \ = \ 
    (-1)^d \LatEnu_{-P({\setminus}|\Delta^\prime|)}(\x) \qedhere
    \]
\end{proof}

\section{Topology of reciprocal domains}
\label{sec:top}

Theorem~\ref{thm:MR1} and Theorem~\ref{thm:atdvc} apply to full-dimensional
(weakly) Cohen-Macaulay complexes in the boundaries of polytopes.  In this
section we discuss what forms these complexes can take.  In~\cite{MR06},
Miller and Reiner gave an example of a full-dimensional Cohen-Macaulay
subcomplex in the boundary of a polytope that is not contractible and hence
not a ball; they argued that, for instance, the Mazur manifold can occur.  The
purpose of this section is to generalize this remark. We refer to
\cite{RourkeSanders} for the basic notions of PL topology.
  
\begin{thm}\label{thm:CMP}
Let $B$ be any PL manifold of dimension $d \ge 5$ such that 
    \begin{compactenum}[\rm (a)]
        \item The natural inclusion $\pi_1(\partial B)\hookrightarrow
            \pi_1(B)$ is surjective, and 
        \item $B$ is homologically trivial, i.e., $\rH_*(B) = \rH_*(B_d)$.
    \end{compactenum}
    Then there exists a $(d+1)$-polytope $P$ and a subcomplex $\widetilde{B}
    \subseteq \B(P)$ such that $\widetilde{B}$ is PL-homeomorphic to 
    $B$. In particular, $\widetilde{B}$ is a full-dimensional 
    weakly Cohen-Macaulay subcomplex of $\partial P$.
  \end{thm}
Any homology manifold $B$ satisfying assumptions (a) and (b) is a
\Defn{homology ball}.
  
\begin{proof}
    By~\cite[Thm.\ 3]{Kervaire}, there is a contractible PL manifold $M$ for
    which $\partial M$ is PL homeomorphic to $\partial B$. Then the gluing of
    $M$ and $B$ along their boundaries is a PL-sphere $S$, since it is PL
    (because $B$ and $M$ are PL), simply connected (by property (a) of $B$ and
    the fact that $M$ is contractible) and has the homology of a sphere (since
    both $M$ and $B$ have the homology of a sphere); consequently, it is a PL
    sphere by the generalized Poincar\'e conjecture~\cite{ZeemanP}.  In
    particular, we have that there exists a subdivision $S'$ of $S$ that is
    combinatorially equivalent to the boundary complex $S''$ of a
    $(d+1)$-polytope $P$. The subcomplex of $S''$ corresponding to $B$ is the
    desired complex $\widetilde{B}$.   
\end{proof}
  
\begin{cor}
    Every contractible PL $d$-manifold $B$, $d\geq 5$ can be realized, up to
    PL homeomorphism, as a full-dimensional weakly Cohen-Macaulay subcomplex
    in the boundary of a $(d+1)$-polytope.  
\end{cor}

This suggests that every PL manifold satisfying (a) and (b) of Theorem~\ref{thm:CMP} is contractible. This is not the case: 

\begin{example}
Let $S$ denote a PL homology sphere that is not $S^d$, such as Poincar\'e's homology sphere, and let $\Delta$ denote any facet of $S$. Then $B:=(S-\Delta)\times [0,1]$ is a homology ball, but homotopy equivalent to $S-\Delta$, which has $\pi_1(S)=\pi_1(S-\Delta)\neq 0$ and is consequently not contractible.
\end{example}

Theorem~\ref{thm:atdvc} applies more generally to subcomplexes in the boundary
of homology manifolds; in this case, we are surprisingly flexible:

\begin{thm}\label{thm:CMH}
    Let $M$ denote any homology manifold with vanishing reduced homology. Then
    there exists a homology ball that contains $M$ as a full-dimensional
    subcomplex of its boundary.  
\end{thm}

\begin{proof}
    Let $D(M,\partial M)$ denote the double of $M$ (that is, the result of
    gluing two manifolds PL homeomorphic to $M$ along their isomorphic
    boundaries). By excision, the complex $D(M,\partial M)$ is a homology
    manifold without boundary which is homologically equivalent to a sphere.
    Thus, the cone over $D(M,\partial M)$ is a homology ball, as desired.
\end{proof}

\bibliographystyle{siam}
\bibliography{AlexanderDualValuations}

\end{document}